\documentclass{article}
\usepackage{amsmath,amssymb}
\usepackage{enumerate}

\newtheorem{theorem}{Theorem}[section]

\newtheorem{proposition}[theorem]{Proposition}

\newtheorem{definition}[theorem]{Definition}

\newtheorem{example}[theorem]{Example}

\newcommand{\id}{\mbox{id}}

\newcommand{\Sym}{\operatorname{Sym}}

\newcommand{\Aut}{\operatorname{Aut}}
\newcommand{\Soc}{\operatorname{Soc}}

\newcommand{\St}{\operatorname{St}}

\newcommand{\Z}{\mathbb{Z}}

\renewcommand{\ker}{\operatorname{Ker}}
\newcommand{\Id}{\operatorname{Id}}
\newcommand{\sym}{\operatorname{Sym}}
\newcommand{\aut}{\operatorname{Aut}}

\newcommand{\soc}{\operatorname{Soc}}
\newcommand{\core}{\operatorname{core}}

\newenvironment{proof}{\par\noindent{\bf Proof.}}{$\qed$\par\bigskip}
\newcommand{\qed}{\enspace\vrule  height6pt  width4pt  depth2pt}
\begin{document}

\title{Solutions of the Yang-Baxter equation associated with a left brace
}

\author{David Bachiller\footnote{Research partially supported by a grant of MICIIN (Spain)
MTM2011-28992-C02-01.}\and Ferran Ced\'o\footnote{Research partially
supported by a grant of MICIIN (Spain) MTM2011-28992-C02-01.}\and
Eric Jespers\footnote{Research supported in part by Onderzoeksraad
of Vrije Universiteit Brussel and Fonds voor Wetenschappelijk
Onderzoek (Belgium).}}
\date{}
\maketitle

\begin{abstract}
Given a left brace $G$, a method to construct all the
involutive, non-degenerate set-theoretic solutions $(Y,s)$ of the YBE,
such that $\mathcal{G}(Y,s)\cong G$ is given. This method depends
entirely on the brace structure of $G$.
\end{abstract}

\section{Introduction}
The quantum Yang-Baxter equation is an important equation coming from theoretical
physics, first appearing in the works of Yang \cite{Yang} and Baxter \cite{Baxter}.
Recall that a solution of the Yang-Baxter equation is a linear map
$R:V\otimes V \longrightarrow V\otimes V$, where $V$ is a vector
space, such that
    $$R_{12}R_{13}R_{23} = R_{23}R_{13}R_{12},$$
where $R_{ij}$ denotes the map $V\otimes V \otimes V \longrightarrow
V\otimes V \otimes V$ acting as $R$ on the $(i,j)$ tensor factor and
as the identity on the remaining factor. A central open problem
is to construct new families of solutions of this equation.
It is this problem which initially motivated the definition of quantum groups,
and one of the reasons of the recent interest in Hopf algebras, see \cite{kassel}.

Note that if $X$ is a basis
of the vector space $V$, then a map $\mathcal{R}:X\times X
\longrightarrow X\times X$, such that
    $$\mathcal{R}_{12}\mathcal{R}_{13}\mathcal{R}_{23} = \mathcal{R}_{23}\mathcal{R}_{13}\mathcal{R}_{12},$$
where $\mathcal{R}_{ij}$ denotes the map $X\times X \times X
\longrightarrow X\times X \times X$ acting as $\mathcal{R}$ on the
$(i,j)$ components and as the identity on the remaining component,
induces a solution of the Yang-Baxter equation. In this case, one
says that $(X,\mathcal{R})$ (or $\mathcal{R}$) is a set-theoretic solution of the quantum
Yang-Baxter equation. Drinfeld, in \cite{drinfeld}, posed the
question of finding these set-theoretic solutions.

A subclass of this type of solutions, the non-degenerate involutive
ones, has received a lot of attention in the last years \cite{CJO,
CJO2, CJR, ESS, Gat, GC, gat-maj, GIVdB, JO, JObook, LYZ, rump1,
rump3}. This class of solutions is not only studied for the
applications of the Yang-Baxter equation
 in physics,
but also for its connection with other topics in mathematics of recent interest:  semigroups of I-type
and Bieberbach groups \cite{GIVdB},
bijective 1-cocycles \cite{ESS}, radical rings \cite{rump3}, triply factorized groups \cite{sysak},
construction of semisimple
minimal triangular Hopf algebras \cite{EtingofGelaki}, regular subgroups of the holomorf
and Hopf-Galois extensions \cite{CR, FCC}, and groups of central type \cite{BenDavid}.

Gateva-Ivanova and Van den Bergh \cite{GIVdB}, and Etingof, Schedler
and Soloviev \cite{ESS} introduced this subclass of solutions and
associated with each solution $(X,r)$ of this type two groups which
are fundamental for its study: the structure group $G(X,r)$, and the
permutation group $\mathcal{G}(X,r)$. In order to study this
class of solutions, Rump in \cite{rump3} introduced a new algebraic
structure called brace. Recall that a left brace is a set $B$
with two operations, $+$ and $\cdot$,
 such that $(B,+)$ is an abelian group, $(B,\cdot)$ is a group and
$$
x\cdot (y+z)+x=x\cdot y+x\cdot z,
$$
for all $x,y,z\in B$. This structure is connected with the solutions of the Yang-Baxter equation because,
besides the group structure of $G(X,r)$ and $\mathcal{G}(X,r)$, any solution $(X,r)$ induces a sum
over these two groups that defines a structure of left brace.

In spite of all this progress, it is still an open problem to construct all the non-degenerate involutive set-theoretic
solutions of the Yang-Baxter equation.
Inspired by \cite{CJR}, and using the new techniques available when we introduce the brace structure,
we separate this problem in two parts:
\begin{enumerate}[1.]
\item Classify all the left braces.
\item For each left brace $G$, construct all the non-degenerate involutive set-theoretic solutions $(Y,s)$
with $\mathcal{G}(Y,s)\cong G$ as left braces, and classify them up to isomorphism.
\end{enumerate}

In this article, we focus on the second problem. We give a translation of the problem
completely in terms of the structure of left brace of $G$. Specifically, we show that the construction of solutions
with $\mathcal{G}(Y,s)\cong G$ is equivalent to find some subgroups of $(G,\cdot)$ and some orbits
of $G$ with respect to an action $\lambda$ that is defined in any left brace.
Note that, when $Y$ is finite, the associated permutation group
$\mathcal{G}(Y,s)\leq \sym_X$ is finite, so the set of subgroups and of orbits of $\mathcal{G}(Y,s)$ is finite,
and our result reduces the question of finding all the solutions of the Yang-Baxter equation
with the same permutation group (which are infinite) to a problem of finding a finite set of objects.

For previous results on the second problem, Ced\'o, Jespers
and del R\'io initiated the study of the problem
presenting a particular construction of this type in \cite[Section 5]{CJR}: for a given solution $(X,r)$, they
give a method to construct a solution $(X^2,r^{(2)})$ over $X^2=X\times X$ such that $\mathcal{G}(X^2,r^{(2)})
\cong \mathcal{G}(X,r)$. This result was then generalized in \cite{BC}, giving a method to construct solutions
 $r^{(n)}$ over $X^n$ for each $n$ such that $\mathcal{G}(X^n,r^{(n)})\cong \mathcal{G}(X,r)$. It was trying to
generalize \cite{BC} that we found the results of the present
article. We also have to mention that \cite[Corollary D]{BenDavid}
provides a homological solution to the second problem, but that
solution is not constructive. Another initial motivation for this
paper was to give an explicit and constructive solution.

The content of the paper is as follows. In Section 2, we
recall some definitions and
known results about left braces and the Yang-Baxter equation that we need in the next sections.
 Then, in Section 3 we present our method of construction of
solutions with respect to some subgroups and some orbits of $G$, and proof that any solution can be obtained
using this construction. It is possible
that two solutions constructed with this method are isomorphic, so
in Section 4 we present a method to detect isomorphism between
solutions, which essentially says that isomorphisms between solutions
are induced by automorphisms of the left brace $G$. So we also manage to translate the isomorphism
problem completely in terms of the brace structure of $G$. Finally, in Section 5,
we suggest a possible way to classify all the non-degenerate set-theoretic solutions of the Yang-Baxter equation
with fixed permutation group through the concept of basic solution, and
we apply the two theorems of Sections 3 and 4 to
find all the basic solutions with permutation group equal
to a brace with multiplicative group isomorphic to $\Z/(p^n)$ for
some prime $p$.

\section{Preliminary results}

\begin{definition}
A left brace is a set with two binary operations, an addition + and a multiplication $\cdot$,
such that $(B,+)$ is an abelian group, $(B,\cdot)$ is a group, and
\begin{eqnarray}\label{bracecond}
&& x\cdot (y+z)+x=x\cdot y+x\cdot z,
\end{eqnarray}
for all $x,y,z\in X$.
\end{definition}

A right brace is defined similarly, but changing the last property by
$(y+z)\cdot x+x=y\cdot x+z\cdot x$. When $B$ is both a left and a right brace, we say that
$B$ is a (two-sided) brace. A morphism between two left braces $B_1$ and $B_2$ is a map
$f:B_1\to B_2$ such that $f(x+y)=f(x)+f(y)$ and $f(x\cdot y)=f(x)\cdot f(y)$.

Now we define some important concepts in the study of left braces.
For $x\in B$, we define a map $\lambda_x\colon
B\longrightarrow B$ by $\lambda_x(y)=xy-x$ for all $y\in
B$. It is known that $\lambda_x$ is an automorphism of the
additive group of $B$ and the map $\lambda\colon B\longrightarrow
\Aut(B,+)$, defined by $x\mapsto \lambda_x$, is a morphism of groups
from the multiplicative group of $B$ to $\Aut(B,+)$. The kernel of this
morphism is called the socle
$$\soc(B):=\{ g\in B\mid \lambda_g=\id \}.$$

We will use in a fundamental way the next result about extensions of
braces with respect to the socle.

\begin{proposition}{
(\cite[Theorem~2.1]{bracesp3})}\label{extensions}  Let $H$ be an abelian group and $B$
be a left brace. Let $\sigma\colon (B,\cdot)\longrightarrow
\aut(H,+)$ be an injective morphism, and $h\colon
(H,+)\longrightarrow (B,+)$ be a surjective morphism. Suppose that
they satisfy $h(\sigma(g)(m))=\lambda_g(h(m))$ for all $g\in B$ and
$m\in H$. Then, the multiplication over $H$ given by
$$x\cdot y:=x+\sigma(h(x))(y)~~\forall x,y\in H,$$
defines a
structure of left brace on $H$ such that $h$ is a morphism of left
braces, $\soc(H)=\ker(h)$ and $H/\soc(H)\cong B$ as left braces.

Two of these structures, determined by $\sigma$, $h$ and $\sigma'$, $h'$ respectively, are isomorphic if
and only if there exists an $F\in \aut(H,+)$ such that
$$
\sigma'(h'(m))=F^{-1}\circ\sigma(h(F(m)))\circ F,
$$
for all $m\in H$.

Conversely, suppose that $G$ is a left brace. Then, the map
$\sigma\colon (G/\soc(G),\cdot)\longrightarrow\aut(G,+)$ induced by
the map $\lambda\colon (G,\cdot)\longrightarrow\aut(G,+)$, and the
natural map $h\colon G\longrightarrow G/\soc(G)$ satisfy the above
properties. $\qed$
\end{proposition}

The importance of this algebraic structure is its relation with a
class of set-theoretic solutions of the Yang-Baxter equation, the
non-degenerate involutive ones. Given a set $X$, recall that a map
$r:X\times X\to X\times X$ is a set-theoretic solution of the
Yang-Baxter equation if
\begin{eqnarray}\label{braid}
&&r_{12}r_{23}r_{12}=r_{23}r_{12}r_{23}, \end{eqnarray} where
$r_{ij}$ denotes the map $X\times X\times X\to X\times X\times X$
acting as $r$ in the $(i,j)$ component and as the identity on the
remaining component\footnote{ To be precise, (\ref{braid}) is the
so-called braid equation. But, $r$ is a solution  of the braid
equation if and only if $\tau\circ r$ is a set-theoretic solution of
the Yang-Baxter equation (defined in the introduction), where
$\tau(x,y)=(y,x)$.}. If we write $r$ as
    $$\begin{array}{cccc} r\colon & X\times
    X&\longrightarrow &X\times X\\
    &(x,y)&\mapsto&(\sigma_x(y), \gamma_y(x)), \end{array}$$
$r$ is said to be involutive if $r^2=\id_{X^2}$. Moreover, it is said to
be non-degenerate if each map $\sigma_{x}$ and $\gamma_{x}$ is
bijective. In what follows, we will only consider non-degenerate
involutive set-theoretic solutions.

\paragraph{Convention:} By a \emph{solution of the YBE} we mean a non-degenerate
involutive set-theoretic solution of the Yang-Baxter equation.

\bigskip
Two groups are very important to study this type of solutions.
The first one, the structure group $G(X,r)$, is defined
by the presentation
$$G(X,r)=\langle x\in X\mid xy=\sigma_x(y)\gamma_y(x),\mbox{ for all } x,y\in X\rangle,$$
where $r(x,y)=(\sigma_x(y),\gamma_y(x))$. In \cite{ESS} it is
proved that the map $X\longrightarrow \mathbb{Z}^{(X)}\rtimes
\Sym_X$ defined by $x\mapsto (x,\sigma_x)$, for all $x\in X$,
extends to a homomorphism
$$\begin{array}{ccc}
G(X,r)&\longrightarrow &\mathbb{Z}^{(X)}\rtimes \Sym_X\\
g&\mapsto&(\pi(g),\phi(g))\end{array}$$ such that the map $\pi\colon
G(X,r)\longrightarrow \mathbb{Z}^{(X)}$ is a bijective $1$-cocycle.
$\mathbb{Z}^{(X)}$ denotes the free abelian group with basis $X$.

It is easy to see that the group $G(X,r)$ with the addition defined
by
$$g+h:=\pi^{-1}(\pi(g)+\pi(h)),$$
for all $g,h\in G(X,r)$, is a left brace. Note that the additive
group of $G(X,r)$ is a free abelian group with basis $X$.
Furthermore, for $g\in G(X,r)$ and $x\in X$, we have that
$\lambda_g(x)=gx-g=\phi(g)(x)\in X$.

The second important group, the permutation group of
$(X,r)$, is defined by
$$\mathcal{G}(X,r):=\{\phi(g)\mid g\in G(X,r)\}\leq\sym_X,$$
or, equivalently,
$$
\mathcal{G}(X,r):=\left\langle\sigma_x\mid x\in X\right\rangle.
$$
Note that $\mathcal{G}(X,r)$ with the addition defined by
$\phi(g)+\phi(h)= \phi(g+h)$, for all $g,h\in G(X,r)$, is a left
brace and the natural projection $G(X,r)\longrightarrow
\mathcal{G}(X,r)$ is a homomorphism of left braces with kernel equal
to the socle of $G(X,r)$.

The next result gives a first way to construct solutions of the YBE.
For references of this result, we can find a homological version of
it in \cite[Corollary D]{BenDavid}, and a first version without a
formal statement in \cite[pages 182--183]{ESS}.

\begin{proposition}\label{BenDavid}
Let $G$ be a left brace, and let $Y$ be a set. Suppose that $h\colon
\Z^{(Y)}\longrightarrow (G,+)$ is a surjective morphism, and that
$\sigma\colon (G,\cdot)\longrightarrow\aut(\Z^{(Y)})$ is an
injective morphism such that $\sigma(g)\mid_Y$ is a bijection of $Y$
for all $g\in G$, and $h(\sigma(g)(m))=\lambda_g(h(m))$ for all
$g\in G$ and $m\in\Z^{(Y)}$. Let $s$ be the map
$$
\begin{array}{cccc}
s\colon & Y\times Y &\longrightarrow & Y\times Y\\
& (x,~y) &\mapsto &(f_{x}(y),~f^{-1}_{f_{x}(y)}(x)),
\end{array}
$$
where $f_x(y):=\sigma(h(x))(y)$. Then $(Y,s)$ is a solution of the
YBE and $\mathcal{G}(Y,s)\cong G$ as left braces. Moreover, any
solution $(Z,t)$ of the YBE with $\mathcal{G}(Z,t)\cong G$ is of
this form.
\end{proposition}
\begin{proof}
By Proposition~\ref{extensions}, the abelian group
$\mathbb{Z}^{(Y)}$ with the multiplication defined by
$$x\cdot y:=x+\sigma(h(x))(y)~~\forall x,y\in \mathbb{Z}^{(Y)},$$
is a left brace and $h$ becomes a homomorphism of left braces. Note
that for $x, y\in Y\subseteq \mathbb{Z}^{(Y)}$, we have that
$$\lambda_x(y)=xy-x=x+\sigma(h(x))(y)-x=\sigma(h(x))(y)=f_x(y).$$
Therefore $(Y,s)$ is a solution of the YBE because it is the
restriction of the solution of the YBE associated with the left
brace $\mathbb{Z}^{(Y)}$ (cf. \cite[Lemma~2]{CJO2}). In fact the
left brace $\mathbb{Z}^{(Y)}$ is equal to the left brace $G(Y,s)$.
Recall that the addition of the left brace
$$\mathcal{G}(Y,s)=\{ \sigma(h(m))|_{Y}\mid m\in \Z^{(Y)}\}$$
is defined by
$\sigma(h(m_1))|_{Y}+\sigma(h(m_2))|_{Y}=\sigma(h(m_1+m_1))|_{Y}$,
for all $m_1,m_2\in \Z^{(Y)}$. Therefore the map $G\longrightarrow
\mathcal{G}(Y,s)$ defined by $g\mapsto \sigma(g)|_{Y}$ is an
isomorphism of left braces.

On the other hand, observe that, if $(Z,t)$ is a solution of the YBE
and  $\eta\colon G\longrightarrow \mathcal{G}(Z,t)$ is an
isomorphism of left braces, then the unique homomorphism
$h\colon\Z^{(Z)}\longrightarrow  (G,+)$ such that
$h(z)=\eta^{-1}(\pi(z))$, for all $z\in Z$, where $\pi\colon
G(Z,t)\longrightarrow \mathcal{G}(Z,t)$ is the natural projection,
is surjective, and the map
$\sigma:(G,\cdot)\longrightarrow\aut(\Z^{(Z)})$, where $\sigma(g)$
is the unique automorphism of $\Z^{(Y)}$ such that
$\sigma(g)(z)=\eta(g)(z)$, for all $z\in Z$, is an injective
homomorphism. Furthermore, for $g\in G$, there exists $m\in
\Z^{(Z)}$ such that $g=h(m)$ and
\begin{eqnarray*}
h(\sigma(g)(z))&=&h(\eta(g)(z))=\eta^{-1}(\pi(\eta(h(m))(z)))\\
&=&\eta^{-1}(\pi(\pi(m)(z)))=\eta^{-1}(\pi(m)\pi(z)-\pi(m))\\
&=&h(m)h(z)-h(m)=gh(z)-g\\
&=&\lambda_g(h(z)),
\end{eqnarray*}
for all $z\in Z$. Therefore $h(\sigma(g)(n))=\lambda_g(h(n))$, for
all $g\in G$ and all $n\in \Z^{(Z)}$. Now we have that
$$\sigma(h(x))(y)=\eta(h(x))(y)=\pi(x)(y),$$
for all $x,y\in Z$. Hence $(Z,t)$ is exactly the same solution of
the YBE that the solution obtained by the given construction using
the maps $h$ and $\sigma$.
\end{proof}

Note that this reduces the problem of finding all the solutions of
the YBE to the problem of finding the two maps $h$ and $\sigma$ with
the required properties. This has the disadvantage that we do not
know in principle how to construct these two maps. Our aim in the
next sections is trying to solve this difficulty.

\section{Construction of solutions}
Given a left brace $G$, we will try to construct all the
involutive non-degenerate set-theoretic solutions of the Yang-Baxter
equation $(Y,s)$ such that $\mathcal{G}(Y,s)\cong G$
as left braces. As we proved in the previous section, this is
equivalent to construct two maps with
some properties.

The next result gives a translation of the problem that only depends
on the brace structure of $G$. Recall that, for a subgroup $K$ of a
group $H$, we define the core of $K$ in $H$ as
$\core(K):=\bigcap_{g\in H} gKg^{-1}$. It is the maximal normal
subgroup of $H$ contained in $K$

\begin{theorem}\label{main}
Let $G$ be a left brace. Let $X$ be a subset of $G$, invariant by
the action $\lambda$, such that $X$ generates $G$ additively. Let
$X=\bigcup_{i\in I} X_i$ be the decomposition of $X$ as disjoint
union of orbits, and choose an element $x_i$ of $X_i$ for any $i\in
I$. Let $H_i$ be the stabilizer $\St(x_i)$ of $x_i$ in $G$. For each
$i\in I$, let $J_i$ be a nonempty set, and let $(K_{i,j})_{j\in
J_{i}}$ be a family of subgroups of $H_i$ such that $\bigcap_{i\in
I}\bigcap_{j\in J_i} \core(K_{i,j})=\{ 1\}$. Let $Y:=\bigcup_{i\in
I}\bigcup_{j\in J_i}G/K_{i,j}$ be the disjoint union of the sets of
left cosets  $G/K_{i,j}$. Then, $(Y,s)$, where $s$ is the map
$$
\begin{array}{cccc}
s\colon& Y\times Y &\longrightarrow & Y\times Y\\
& (g_1K_{i_1,j_1},~g_2K_{i_2,j_2}) &\mapsto &(f_{g_1K_{i_1,j_1}}(g_2K_{i_2,j_2}),~f^{-1}_{f_{g_1K_{i_1,j_1}}(g_2K_{i_2,j_2})}(g_1K_{i_1,j_1})),
\end{array}
$$
 with
$f_{g_1K_{i_1,j_1}}(g_2K_{i_2,j_2})=\lambda_{g_1}(x_{i_1})g_2K_{i_2,j_2}$,
is a solution of the YBE such that $\mathcal{G}(Y,s)\cong G$ as left
braces.

Moreover, any solution $(Z,t)$, with $\mathcal{G}(Z,t)\cong G$ as
left braces, is isomorphic to one of this form.
\end{theorem}

\begin{proof}
First we shall prove that $(Y,s)$ is a solution of the YBE such that
$\mathcal{G}(Y,s)\cong G$ as left braces. We define $h\colon
Y\longrightarrow X$ by $h(gK_{i,j})=\lambda_g(x_i)$, and define
$\sigma\colon G\longrightarrow \sym_{Y}$ to be the natural action of
$G$ on $Y$ given by left multiplications on the cosets in
$G/K_{i,j}$; i.e. $\sigma(g)(xK_{i,j}):=gxK_{i,j}$. Note that $h$
extends to a unique epimorphism $h\colon
\mathbb{Z}^{(Y)}\longrightarrow (G,+)$ and, for each $g\in G$,
$\sigma(g)$ extends to a unique automorphism of $\mathbb{Z}^{(Y)}$.
Since
\begin{eqnarray*}h(\sigma(g)(xK_{i,j}))&=&h(gxK_{i,j})=\lambda_{gx}(x_i)\\
&=&\lambda_{g}(\lambda_{x}(x_i))=\lambda_{g}(h(xK_{i,j})),
\end{eqnarray*}
we have that $h(\sigma(g)(m))=\lambda_g(h(m))$, for all $g\in G$ and
$m\in \mathbb{Z}^{(Y)}$. Let $g\in G$ be an element such that
$\sigma(g)=\id$. Hence $gxK_{i,j}=xK_{i,j}$, for all $x\in G$, $i\in
I$ and $j\in J_i$. Thus $g\in \core(K_{i,j})$, for all $i\in I$ and
$j\in J_i$. Since $\bigcap_{i\in I}\bigcap_{j\in J_i}
\core(K_{i,j})=\{ 1\}$, we have $g=1$. Therefore $\sigma$ is
injective. Hence, by Proposition~\ref{BenDavid}, $(Y,s)$ is a
solution of the YBE and $\mathcal{G}(Y,s)\cong G$ as left braces.

Let $(Z,t)$ be a solution of the YBE such that
$\mathcal{G}(Z,t)\cong G$ as left braces. Let $\eta\colon
\mathcal{G}(Z,t)\longrightarrow G$ be an isomorphism of left
braces. Let $h=\eta\circ \phi$, where $\phi\colon
G(Z,t)\longrightarrow \mathcal{G}(Z,t)$ is the natural projection.
Then $h(Z)$ is a subset of $G$, invariant by $\lambda$ which
generates $G$ additively. Let $X=h(Z)$.

We also have an injective morphism $\sigma:
(G,\cdot)\longrightarrow\aut(\Z^{(Z)})$, such that
$\sigma(g)(z)=\eta^{-1}(g)(z)$, for all $g\in G$ and $z\in Z$.
Therefore $Z$ is a left $G$-set with the action induced by $\sigma$.
Let $a\in G$ and $z\in Z$. Let $g\in G(Z,t)$ such that
$\phi(g)=\eta^{-1}(a)$. We have
\begin{eqnarray*}
h(\sigma(a)(z))&=&\eta(\phi(\eta^{-1}(a)(z)))=\eta(\phi(\phi(g)(z)))\\
&=&\eta(\phi(gz-g))=\eta(\phi(g))\eta(\phi(z))-\eta(\phi(g))\\
&=&ah(z)-a=\lambda_a(h(z)).
\end{eqnarray*}
Therefore the restriction $h|_Z\colon Z\longrightarrow X$ of $h$ is
a $G$-map.

 Let $X=\bigcup_{i\in I} X_i$ be the decomposition of $X$
as disjoint union of orbits under the action $\lambda$. Since $h|_Z$
is a surjective $G$-map, for all $i\in I$, the action $\sigma$
splits $(h|_Z)^{-1}(X_i)$ into orbits:
$(h|_Z)^{-1}(X_i)=\bigcup_{j\in J_i} Z_{i,j}$ and $h(Z_{i,j})=X_i$.
So we have $Z=\bigcup_{i\in I}\bigcup_{j\in J_i} Z_{i,j}$, where
$h(Z_{i,j})=X_i$ for all $i,j$.

For each $i\in I$, we choose an element $x_i\in X_i$, and for each
$j\in J_i$, we choose $z_{i,j}\in Z_{i,j}$ such that
$h(z_{i,j})=x_i$. Note that  $\St(z_{i,j})\leq \St(x_i))\leq G$,
since any $g\in \St(z_{i,j})$ satisfies $\sigma(g)(z_{i,j})=z_{i,j}$
and, applying $h$, we obtain
$x_i=h(z_{i,j})=h(\sigma(g)(z_{i,j}))=\lambda_g(h(z_{i,j}))=\lambda_g(x_i)$,
so $g\in \St(x_i)$.

Recall that the maps $f_i\colon X_{i}\rightarrow G/\St(x_i)$ and
$f_{i,j}\colon Z_{i,j}\rightarrow G/\St(z_{i,j})$ defined by
$f_{i}(\lambda_g(x_i))= g\St(x_i)$ and $f_{i,j}(\sigma(g)(z_{i,j}))=
g\St(z_{i,j})$ are isomorphisms of $G$-sets.
 Then,
$h\mid_{Z_{i,j}}$ is determined by the canonical projection
$$
\begin{array}{cccc}
\pi \colon &G/\St(z_{i,j})&\to &G/\St(x_i)\\
&t\St(z_{i,j})&\mapsto &t\St(x_i),
\end{array}
$$
that is, \begin{eqnarray*}
h(f_{i,j}^{-1}(g\St(z_{i,j})))&=&h(\sigma(g)(z_{i,j}))=\lambda_g(h(z_{i,j}))\\
&=&
\lambda_g(x_{i})=f_i^{-1}(g\St(x_i))\\
&=&f_i^{-1}(\pi(g\St(z_{i,j}))) \end{eqnarray*}

Note that $\sigma(h)(z)=z$, for all $z\in Z$, if and only if
$hg\St(z_{i,j})=g\St(z_{i,j})$, for all $g\in G$ and all $i,j$, i.e.
if and only if $h\in\bigcap_{i,j} \core(\St(z_{i,j}))$. Since
$\sigma$ is injective, we have $\bigcap_{i,j} \core(\St(z_{i,j}))=\{
1\}$. Defining $H_i:=\St(x_i)$ and $K_{i,j}:=\St(z_{i,j})$, let
$(Y,s)$ the solution defined as in the statement of
Theorem~\ref{main}. We shall prove that $(Y,s)\cong (Z,t)$. Let
$f\colon Z\longrightarrow Y$ be the map defined by
$f(z)=f_{i,j}(z)$, for all $z\in Z_{i,j}$. Clearly $f$ is bijective.
Let $x,z\in Z$. We may assume that $z\in Z_{i,j}$ and $x\in
Z_{p,q}$. Therefore there exist $g_1,g_2\in G$ such that
$z=\sigma(g_1)(z_{i,j})$ and $x=\sigma(g_2)(z_{p,q})$. We have
\begin{eqnarray*}
f(\phi(x)(z))&=&f(\sigma(h(x))(z))=f_{i,j}(\sigma(h(x))(\sigma(g_1)(z_{i,j})))\\
&=&f_{i,j}(\sigma(h(x)g_1)(z_{i,j}))=h(x)g_1\St(z_{i,j})\\
&=&h(x)g_1K_{i,j}=h(\sigma(g_2)(z_{p,q})g_1K_{i,j}\\
&=&\lambda_{g_2}(h(z_{p,q})g_1K_{i,j}=\lambda_{g_2}(x_p)g_1K_{i,j}\\
&=&f_{g_2K_{p,q}}(g_1K_{i,j})\\
&=&f_{f(x)}(f(z)).
\end{eqnarray*}
Therefore $f$ is an isomorphism of solutions of the YBE.
\end{proof}

Summarizing, given a left brace $G$, to construct all the solutions
$(Y,s)$ of the YBE such that $\mathcal{G}(Y,s)\cong G$ as left
braces one can proceed as follows:
\begin{enumerate}
\item Find the decomposition of $G$ as disjoint union of orbits, $G=\bigcup_{i\in K} G_i$, by the action $\lambda\colon (G,\cdot)\longrightarrow\aut(G,+)$.
Then choose one element $x_i$ in each orbit $G_i$ for all $i\in
K$.

\item Find all the subsets $I$ of $K$ such that the subset
$X=\bigcup_{i\in I}G_i$ generates the additive group of $G$.

\item Given such an $X$, find for each $i\in I$ a non-empty family $(K_{i,j})_{j\in
J_i}$ of subgroups of $\St(x_i)$ such that $\bigcap_{i,j}
\core(K_{i,j})=\{ 1\}$. Note that the $K_{i,j}$ could be equal for
different $(i,j)$.

\item Construct a solution as in the statement of
Theorem~\ref{main} using the families $(K_{i,j})_{j\in J_i}$, for
$i\in I$.
\end{enumerate}

Note that by Theorem~\ref{main}, any solution $(Y,s)$ of the YBE
such that $\mathcal{G}(Y,s)\cong G$ as left braces is isomorphic to
one constructed in this way. It could happen that different
solutions of the YBE constructed in this way from a left brace $G$ are in
fact isomorphic. In the next section we characterize when two of
these solutions are isomorphic.


\section{Isomorphism of solutions}

We begin with a left brace $G$ with orbits $\{G_i\}_{i\in I}$ under
the action $\lambda\colon G\rightarrow \Aut(G,+)$ defined as usual
$\lambda(g)=\lambda_g$. Choose an element $x_i\in G_i$ in each orbit
$G_i$. Let $I_1, I_2$ be subsets of $I$, such that $X_1=\cup_{i\in
I_1}G_i$ and $X_2=\cup_{j\in I_2}G_j$ satisfy $G=\langle
X_1\rangle_+=\langle X_2\rangle_+$. For each $i\in I_1$, let $\{
K_{i,k} \}_{k\in A_i}$ be a non-empty family of subgroups of
$\St_G(x_i)$ such that
$$\bigcap_{i\in I_1}\bigcap_{k\in A_i}\core(K_{i,k})=\{ 1\}.$$
For each $j\in I_2$, let $\{ L_{j,l} \}_{l\in B_j}$ be a non-empty
family of subgroups of $\St_G(x_j)$ such that
$$\bigcap_{j\in I_2}\bigcap_{l\in B_j}\core(L_{j,l})=\{ 1\}.$$
Let $Y_1$ be the disjoint union of the family of left $G$-sets
$G/K_{i,k}$, for $i\in I_1$ and $k\in A_i$. Let $Y_2$ be the
disjoint union of the family of left $G$-sets $G/L_{j,l}$, for $j\in
I_2$ and $l\in B_j$. Let $s_1\colon Y_1^2\rightarrow Y_1^2$
and $s_2\colon Y_2^2\rightarrow Y_2^2$ be maps defined by
\begin{eqnarray*}
s_1(g_1K_{i_1,k_1},g_2K_{i_2,k_2})&=&(\lambda_{g_1}(x_{i_1})g_2K_{i_2,k_2},\lambda_{\lambda_{g_1}(x_{i_1})g_2}(x_{i_2})^{-1}g_1K_{i_1.k_1}),\\
s_2(g_1L_{j_1,l_1},g_2L_{j_2,l_2})&=&(\lambda_{g_1}(x_{j_1})g_2L_{j_2,l_2},\lambda_{\lambda_{g_1}(x_{j_1})g_2}(x_{j_2})^{-1}g_1L_{j_1.l_1}).
\end{eqnarray*}
We know that $(Y_1,s_1)$ and $(Y_2,s_2)$ are solutions of the YBE
(and any solution of the YBE is isomorphic to one constructed in
this way). We shall characterize when $(Y_1,s_2)$ and $(Y_2,s_2)$
are isomorphic in the following result.

\begin{theorem}\label{isomorphism}
The solutions $(Y_1,s_1)$ and $(Y_2,s_2)$ are isomorphic if and only
if there exist an automorphism $\psi$ of the left brace $G$, a
bijective map $\alpha\colon I_1\rightarrow I_2$, a bijective map
$\beta_i\colon A_i\rightarrow B_{\alpha(i)}$ and $z_{i,k}\in G$,
for each $i\in I_1$ and $k\in A_i$, such that
$$\psi(x_i)=\lambda_{z_{i,k}}(x_{\alpha(i)})\quad\mbox{and}\quad \psi(K_{i,k})=z_{i,k}L_{\alpha(i),\beta_i(k)}z_{i,k}^{-1},$$
for all $i\in I_1$ and $k\in A_i$.
\end{theorem}

\begin{proof}
Suppose that there exist an automorphism $\psi$ of the left brace
$G$, a bijective map $\alpha\colon I_1\rightarrow I_2$, a bijective
map $\beta_i\colon A_i\rightarrow B_{\alpha(i)}$
 and $z_{i,k}\in G$, for $i\in I_1$, and for $k\in A_i$, such that
$$\psi(x_i)=\lambda_{z_{i,k}}(x_{\alpha(i)})\quad\mbox{and}\quad \psi(K_{i,k})=z_{i,k}L_{\alpha(i),\beta_i(k)}z_{i,k}^{-1}.$$
Observe that we also have $\psi(\lambda_g(x_i))=\lambda_{\psi(g)z_{i,k}}(x_{\alpha(i)})$ for every $g$, because
$\psi$ is a morphism of braces.
We define $F\colon Y_1\rightarrow Y_2$ by
$F(gK_{i,k})=\psi(g)z_{i,k}L_{\alpha(i),\beta_i(k)}$, for all $i\in
I_1$, $k\in A_i$ and $g\in G$. Since
$\psi(K_{i,k})=z_{i,k}L_{\alpha(i),\beta_i(k)}z_{i,k}^{-1}$, $F$ is
well defined. It is easy to check that $F$ is an isomorphism of the
solutions $(Y_1,s_1)$ and $(Y_2,s_2)$.

Conversely, suppose that there exists an isomorphism $F\colon
Y_1\rightarrow Y_2$ of the solutions $(Y_1,s_1)$ and $(Y_2,s_2)$. We
can write
$F(gK_{i,k})=\varphi(gK_{i,k})L_{\alpha(g,i,k),\beta(g,i,k)}$, for
some maps $\varphi\colon Y_1\rightarrow G$, $\alpha\colon
Y_1\rightarrow I_2$ and $\beta \colon Y_1\rightarrow \bigcup_{j\in
I_2}B_{j}$. We shall prove that $\alpha(g,i,k)=\alpha(1,i,k')$ and
$\beta(g,i,k)=\beta(1,i,k)$, for all $g\in G$, $i\in I_1$ and
$k,k'\in A_i$. Since $F$ is a morphism of solutions of the YBE, we
have
\begin{eqnarray}\label{F}
&&F(\lambda_{g_1}(x_{i_1})g_2K_{i_2,k_2})
=\lambda_{\varphi(g_1K_{i_1,k_1})}(x_{\alpha(g_1,i_1,k_1)})F(g_2K_{i_2,k_2}),\end{eqnarray}
for all $g_1,g_2\in G$, $i_1,i_2\in I_1$,  $k_1\in A_{i_1}$
and $k_2\in A_{i_2}$. Hence
\begin{eqnarray*}\lefteqn{\varphi(\lambda_{g_1}(x_{i_1})g_2K_{i_2,k_2})L_{\alpha(\lambda_{g_1}(x_{i_1})g_2,i_2,k_2),\beta(\lambda_{g_1}(x_{i_1})g_2,i_2,k_2)}}
\\
&&=\lambda_{\varphi(g_1K_{i_1,k_1})}(x_{\alpha(g_1,i_1,k_1)})\varphi(g_2K_{i_2,k_2})L_{\alpha(g_2,i_2,k_2),\beta(g_2,i_2,k_2)},\end{eqnarray*}
for all $g_1,g_2\in G$, $i_1,i_2\in I_1$,  $k_1\in A_{i_1}$
and $k_2\in A_{i_2}$. Thus
$\alpha(\lambda_{g_1}(x_{i_1})g_2,i_2,k_2)=\alpha(g_2,i_2,k_2)$ and
$\beta(\lambda_{g_1}(x_{i_1})g_2,i_2,k_2)=\beta(g_2,i_2,k_2)$. Since
$G=\langle X_1\rangle_+$ and $X_1$ is $G$-invariant (by the action
$\lambda$), we know that $X_1$ also generates the multiplicative
group of $G$. Therefore $\alpha(g_2,i_2,k_2)=\alpha(1,i_2,k_2)$ and
$\beta(g_2,i_2,k_2)=\beta(1,i_2,k_2)$. Note also that
$$\lambda_{\varphi(g_1K_{i_1,k_1})}(x_{\alpha(g_1,i_1,k_1)})F(g_2K_{i_2,k_2})=\lambda_{\varphi(g_1K_{i_1,k})}(x_{\alpha(g_1,i_1,k)})F(g_2K_{i_2,k_2}),$$
for all $g_1,g_2\in G$, $i_1,i_2\in I_1$,  $k_1,k\in A_{i_1}$ and
$k_2\in A_{i_2}$. Since $\bigcap_{j\in I_2}\bigcap_{l\in
B_j}\core(L_{j,l})=\{ 1\}$, we have that
$$\lambda_{\varphi(g_1K_{i_1,k_1})}(x_{\alpha(g_1,i_1,k_1)})=\lambda_{\varphi(g_1K_{i_1,k})}(x_{\alpha(g_1,i_1,k)}),$$
for all $g_1\in G$, $i_1\in I_1$ and  $k_1,k\in A_{i_1}$. Therefore
$x_{\alpha(g_1,i_1,k_1)},x_{\alpha(g_1,i_1,k)}\in
G_{\alpha(1,i_1,k)},$ for all $g_1\in G$, $i_1\in I_1$ and $k_1,k\in
A_{i_1}$ and thus $\alpha(g,i,k)=\alpha(1,i,k')$, for all $g\in G$,
$i\in I_1$ and $k,k'\in A_{i}$. For each $i\in I_1$ we choose an
element $k_i\in A_{i}$. Since $F$ is bijective, the map
$I_1\rightarrow I_2$ defined by $i\mapsto \alpha(1,i,k_i)$ is
bijective and for each $i\in I_1$ the map $A_i\rightarrow
B_{\alpha(1,i,k_i)}$ defined by $k\mapsto \beta(1,i,k)$ is
bijective. We shall see that there exists an automorphism $\psi$ of
the left brace $G$ such that
$$\psi(\lambda_g(x_i))=\lambda_{\psi(g)\varphi(K_{i,k_i})}(x_{\alpha(1,i,k_i)})\quad\mbox{and}\quad \psi(K_{i,k})=\varphi(K_{i,k_i})L_{\alpha(1,i,k_i),\beta(1,i,k)}\varphi(K_{i,k_i})^{-1},$$
for all $g\in G$, $i\in I_1$ and $k\in A_i$. Let
$1=\lambda_{g_1}(x_{i_1})^{\varepsilon_1}\cdots
\lambda_{g_m}(x_{i_m})^{\varepsilon_m}$, for some $g_1,\dots ,g_m\in
G$, $i_1,\dots ,i_m\in I_1$ and $\varepsilon_1,\dots
,\varepsilon_m\in \{ 1,-1\}$. By (\ref{F}), we have
\begin{eqnarray*}
F(gK_{i,k})&=&F(\lambda_{g_1}(x_{i_1})^{\varepsilon_1}\cdots
\lambda_{g_m}(x_{i_m})^{\varepsilon_m}gK_{i,k})\\
&=&\lambda_{\varphi(g_1K_{i_1,k_{i_1}})}(x_{\alpha(1,i_1,k_{i_1})})^{\varepsilon_1}F(\lambda_{g_2}(x_{i_2})^{\varepsilon_2}\cdots
\lambda_{g_m}(x_{i_m})^{\varepsilon_m}gK_{i,k})\\
&=&\lambda_{\varphi(g_1K_{i_1,k_{i_1}})}(x_{\alpha(1,i_1,k_{i_1})})^{\varepsilon_1}\cdots
\lambda_{\varphi(g_mK_{i_m,k_{i_m}})}(x_{\alpha(1,i_m,k_{i_m})})^{\varepsilon_m}F(gK_{i,k}),
\end{eqnarray*}
for all $g\in G$, $i\in I_1$ and $ k\in A_i$. Since $\bigcap_{j\in
I_2}\bigcap_{l\in B_j}\core(L_{j,l})=\{ 1\}$, we have that
$\lambda_{\varphi(g_1K_{i_1,k_{i_1}})}(x_{\alpha(1,i_1,k_{i_1})})^{\varepsilon_1}\cdots
\lambda_{\varphi(g_mK_{i_m,k_{i_m}})}(x_{\alpha(1,i_m,k_{i_m})})^{\varepsilon_m}=1$.
Therefore there exists a unique morphism $\psi\colon G\rightarrow G$
of multiplicative groups such that
$\psi(\lambda_g(x_i))=\lambda_{\varphi(gK_{i,k_{i}})}(x_{\alpha(1,i,k_i)})$.
Since $X_1$ generates the multiplicative group of $G$, by (\ref{F})
one can see that
$$\varphi(gK_{i,k})L_{\alpha(1,i,k_i),\beta(1,i,k)}=F(gK_{i,k})=\psi(g)\varphi(K_{i,k})L_{\alpha(1,i,k_i),\beta(1,i,k)}.$$
Therefore, since $L_{\alpha(1,i,k_i),\beta(1,i,k)}\subseteq
\St(x_{\alpha(1,i,k_i)})$, we have
$$\lambda_{\varphi(gK_{i,k})}(x_{\alpha(1,i,k_i)})=\lambda_{\psi(g)\varphi(K_{i,k})}(x_{\alpha(1,i,k_i)}).$$
Hence
$\psi(\lambda_g(x_i))=\lambda_{\psi(g)\varphi(K_{i,k_i})}(x_{\alpha(1,i,k_i)})$.
Now we have that
\begin{eqnarray*}
\psi(g+x_i)&=&\psi(g\lambda_{g^{-1}}(x_i))=\psi(g)\psi(\lambda_{g^{-1}}(x_i))\\
&=&\psi(g)\lambda_{\psi(g)^{-1}\varphi(K_{i,k_i})}(x_{\alpha(1,i,k_i)})\\
&=&\psi(g)\lambda_{\psi(g)^{-1}}(\lambda_{\varphi(K_{i,k_i})}(x_{\alpha(1,i,k_i)}))\\
&=&\psi(g)+\lambda_{\varphi(K_{i,k_i})}(x_{\alpha(1,i,k_i)})\\
&=&\psi(g)+\psi(\lambda_{1}(x_i))=\psi(g)+\psi(x_i).
\end{eqnarray*}
Now it is easy to see that $\psi$ is a morphism of left braces.
Since $F$ is bijective and $F(gg'K_{i,k})=\psi(g)F(g'K_{i,k})$, it
follows that $\psi$ is bijective. Furthermore $g\in K_{i,k}$ if and
only if
\begin{eqnarray*}
\varphi(K_{i,k})L_{\alpha(1,i,k_i),\beta(1,i,k)}&=&F(K_{i,k})=F(gK_{i,k})=\psi(g)F(K_{i,k})\\
&=&\psi(g)\varphi(K_{i,k})L_{\alpha(1,i,k_i),\beta(1,i,k)}.
\end{eqnarray*}
Therefore the result follows.
\end{proof}

Summarizing, the last theorem says that two solutions constructed as
in Theorem~\ref{main} are isomorphic if we can find an automorphism
of the left brace $G$ that brings each $K_{i,k}$ to one $L_{j,l}$,
taking into account that maybe the $L_{j,l}$'s are permuted (that is
the reason for the $\alpha$ and $\beta_i$ maps), and that maybe we
have chosen another element of the orbit in the process (that is the
reason why the image $x_i$ is $\lambda_{z_{i,k}}(x_{\alpha(i)})$ and
not just $x_{\alpha(i)}$, and it is also the reason why the
$L_{\alpha(i),\beta_i(k)}$ is conjugated by $z_{i,k}$).

The following is an example of how to use Theorems \ref{main} and
\ref{isomorphism} to compute all the finite solutions
associated to a given finite left brace up to isomorphism. We
use the easiest examples of braces: trivial braces of order $p$,
where $p$ is a prime.

\begin{example}\label{bracesp}
 Consider the trivial brace over $G=\Z/(p)$. Then, the orbits are $\{\alpha\}$ for every $\alpha\in\Z/(p)$.
 Since any orbit has one element, then $\St(\alpha)=G$, and the possible $K_{i,j}$'s in this case are $0$ and $G$.

Let $X$ be a subset of $\Z/(p)$ with at least a nonzero
element.
 Let $K_{\alpha,j}=G$ for $\alpha\in X$ and $j\in\{1,\dots,k_\alpha\}$, and let $K'_{\alpha,k}=0$ for
 $\alpha\in X$ and $k\in\{1,\dots,m_\alpha\}$, where $k_{\alpha}$ and $m_{\alpha}$ are non-negative integers such that $k_{\alpha}+m_{\alpha}>0$.
 Write $G/K_{\alpha,j}=\{y_{\alpha,j}\}$, and
 $G/K'_{\alpha,k}=\{y^1_{\alpha,k},\dots,y^p_{\alpha,k}\}$, where
 $y^{l}_{\alpha,k}=l+K'_{\alpha,k}$. Assume that at least one $m_{\alpha}$ is positive. Then the corresponding solution
 of the YBE is $(Y,r)$, where
$$
Y=\bigcup_{\alpha\in X}\left(\bigcup_{1\leq j\leq k_{\alpha}}\{
y_{\alpha,j}\}\right)\cup \left(\bigcup_{1\leq k\leq m_{\alpha}}
\{y^1_{\alpha,k},\dots ,y^p_{\alpha,k}\}\right)
$$
and $r(x,y)=(\sigma_x(y),\sigma^{-1}_{\sigma_{x}(y)}(x))$, with the
sigma maps given by
$$
\sigma_{y_{\alpha,j}}=\sigma_{y^l_{\alpha,k}}=\tau^\alpha\text{, for
all } \alpha\in X,\text{ for all } j,k \text{ and for all }
l\in\{1,\dots,p\},
$$
where $\tau\in \Sym_Y$ is the product of all the cycles of length
$p$ $(y^1_{\alpha,k},y^2_{\alpha,k},\dots,y^p_{\alpha,k})$ for any
$\alpha\in X$ and $k\in\{1,\dots,m_\alpha\}$.

Finally observe that, in this case, $\aut(G,+,\cdot)=\aut(G,+)\cong (\Z/(p))^*$,
and the effect of an automorphism of $G$ over a solution is to change
$\sigma_{y_{\alpha,j}}=\sigma_{y^l_{\alpha,k}}=\tau^\alpha$ to the isomorphic solution
$\sigma_{y_{\alpha,j}}=\sigma_{y^l_{\alpha,k}}=\tau^{A\alpha}$, where $A\in(\Z/(p))^*$.
\end{example}


\section{Basic solutions and examples}
It seems difficult to apply Theorems \ref{main} and \ref{isomorphism} to construct and
classify all the solutions $(Y,s)$ of the YBE with isomorphic left brace $\mathcal{G}(Y,s)$
because there are a lot of different solutions (for instance, there is a lot of freedom
choosing the subgroups $K_{i,j}$). Maybe an
easier problem is to classify a smaller class of solutions from
which we can recover all the other solutions. For example, consider
the following ways to obtain new solutions from a given one:
\begin{enumerate}
 \item Adding a new $K$ (and the corresponding $X_i$, if necessary) to $\{K_{i,j}\}$: this gives a
 solution because the intersection of the cores remains trivial when we add another subgroup, and
 $X$ still generates $G$ if we add a new $X_i$.
 \item Changing a $K_{i,j}$ by a $K\leq K_{i,j}$: this gives a solution because the intersection
 of the cores remains trivial when we change one of the $K_{i,j}$ by a smaller subgroup, and we are
 not changing $X$.
\end{enumerate}

Thus maybe a good definition for basic solution is the converse of
constructions 1 and 2: a solution in which is impossible to take out
a $K_{i,j}$ or to change a $K_{i,j}$ by a $K_{i,j}\leq K'_{i,j}\leq
H_i$ without losing the property of being a solution. From these
solutions, and using construction 1 and 2, we can recover all the
other solutions. Note that these class of solutions can be also
described as the solutions $(X,r)$ such that every surjective
morphism of solutions $(X,r)\to (Y,s)$, to another solution $(Y,s)$
with $\mathcal{G}(X,r)\cong G(Y,s)$, is an isomorphism.

\begin{example}
 In the case of Example \ref{bracesp}, the basic solutions are, up to automorphism of $G$:
 \begin{enumerate}
  \item $X_1=\{1\}$, $K_{1,1}=0$. It corresponds to the solution over the set $\{ 1,\dots ,p\}$
  given by $\sigma_1=\dots=\sigma_p=(1,2,\dots,p)$.
  \item $\{X_1=\{0\}, X_2=\{1\}\}$, $\{K_{1,1}=0, K_{2,1}=G\}$. It corresponds to the solution over
   the set $\{ 1,\dots, p+1\}$ given by $\sigma_1=\dots=\sigma_p=\id$, $\sigma_{p+1}=(1,2,\dots,p)$.
 \end{enumerate}
\end{example}

So, when we only consider the class of basic solutions, the
classification for braces of order $p$ turns out to be much easier.
This example brings some hope that maybe it is possible to classify
the basic solutions associated to any finite left brace $G$.

We are not able to solve this problem in general because our
knowledge of the brace structure is still limited. Nevertheless, we
will classify now the basic solutions for some concrete examples of
finite left braces. Note that good candidates are multiplicative
groups such that the intersection of all their non-trivial subgroups
is non-trivial, because in this case we always need some $K_{i,j}$
equal to $\{ 1\}$, and this condition might restrict the
possible basic solutions. Groups with this property must be
$p$-groups, and they have been classified in \cite[Theorem 5.4.10
(ii)]{Gorenstein}: they are the cyclic $p$-groups and the
generalized quaternion groups $Q_{2^m}$ of order a power of $2$.

We shall study first the cyclic case.
We need a complete knowledge of all the possible brace structures with
cyclic multiplicative group. First of all, the next proposition describes
the additive group in this case.

\begin{proposition}{(converse result to Rump's classification)}
Let $G$ be a left brace with $(G,\cdot)\cong \Z/(p^n)$ and $n\geq 3$. Then, $(G,+)\cong \Z/(p^n)$.
\end{proposition}
\begin{proof}
We prove it by induction over $n$. The case $n=3$ is true by the classification of braces of order $p^3$ in
 \cite{bracesp3}. For $n>3$, assume that
 it is true for $n-1$. Then, since $(G,\cdot)$ is abelian, $G$ is a two-sided brace, so $\Soc(G)\neq 0$ by \cite[Proposition 3]{CJO2}.
 Thus there exists an
 element $x\neq 0$ in $\Soc(G)$ of multiplicative order $p$. Note that in this case $\Soc(G)=\operatorname{Fix}(\lambda_\xi)$, where $\xi$ is
 a multiplicative generator of $G$. This implies that $\langle x\rangle_\cdot=\langle x\rangle_+$ is an ideal of $G$.
  Then $G/\langle x\rangle$ is a brace of order $p^{n-1}$ with $(G/\langle x\rangle,\cdot)\cong \Z/(p^{n-1})$. By induction hypothesis,
 $(G/\langle x\rangle,+)\cong \Z/(p^{n-1})$. This restricts the possible additive groups of $G$ to only two cases: $\Z/(p^n)$ or
 $\Z/(p)\times\Z/(p^{n-1})$. If $(G,+)\cong\Z/(p^n)$, we are done, so assume $(G,+)\cong \Z/(p)\times\Z/(p^{n-1})$ to arrive
 to a contradiction.

 The element $x$ in $\Z/(p)\times\Z/(p^{n-1})$ must be of the form $(\alpha,p^{n-2}\beta)$ because it has additive order
 equal to $p$. Moreover, it must be a fix element of the automorphism $\lambda_\xi$. An automorphism of $\Z/(p)\times\Z/(p^{n-1})$
 can be written in the matrix form
 $\begin{pmatrix}
   a& b\\
   p^{n-2}c& d\\
  \end{pmatrix}$, with $ad\not\equiv 0\pmod{p}$,
  where the entries in the first row are elements of $\Z/(p)$, and the entries in the
  second one are elements of $\Z/(p^{n-1})$. It is a compact way to express that any endomorphism
  $f$ of $\Z/(p)\times\Z/(p^{n-1})$ satisfies $f(x,y)=xf(1,0)+yf(0,1)$, so it is determined
  by $f(1,0)=(a,p^{n-2}c)$ (which has to be an element of order $p$), and $f(0,1)=(b,d)$.
  The condition $ad\not\equiv 0\pmod{p}$ ensures that $f$ is bijective.

  Thus the order of $\Aut(\Z/(p)\times\Z/(p^{n-1})$ is $(p-1)\cdot p\cdot p\cdot(p^n-p^{n-1})=p^{n+1}(p-1)^2$.
  Then, a $p$-Sylow subgroup of $\Aut(\Z/(p)\times\Z/(p^{n-1})$ is
  $$\left\lbrace\begin{pmatrix}
   1& b\\
   p^{n-2}c& 1+pd\\
  \end{pmatrix}:~ b\in\Z/(p),~ c,d\in\Z/(p^{n-1})\right\rbrace.$$
Since $\lambda_\xi$ has prime power order, after a suitable conjugation, it can be written as
$\begin{pmatrix}
   1& b\\
   p^{n-2}c& 1+pd\\
  \end{pmatrix}$. So a fixed element $x=(\alpha,p^{n-2}\beta)$ must satisfy
  $$
  \begin{pmatrix}
   \alpha\\
   p^{n-2}\beta
  \end{pmatrix}=
  \begin{pmatrix}
   1& b\\
   p^{n-2}c& 1+pd\\
  \end{pmatrix}
  \begin{pmatrix}
   \alpha\\
   p^{n-2}\beta
  \end{pmatrix}=
  \begin{pmatrix}
   \alpha\\
   p^{n-2}c\alpha+p^{n-2}\beta
  \end{pmatrix}.
  $$
  This is satisfied if $c=0$ or if $\alpha=0$. If $\alpha=0$, then $x=(0,p^{n-2}\beta)$. But
  $$
  (\Z/(p)\times\Z/(p^{n-1}))/\langle (0,p^{n-2}\beta)\rangle\cong \Z/(p)\times\Z/(p^{n-2}),
  $$
  a contradiction. On the other hand, if $c=0$, then
  $$
  \lambda_\xi^k=\begin{pmatrix}
                 1& kb\\
                 0& (1+pd)^k
                \end{pmatrix}
  $$
  $$
  \Id+\lambda_\xi+\lambda_\xi^2+\cdots +\lambda_\xi^{j-1}=
  \begin{pmatrix}
   j& \binom{j}{2} b\\
   0& \sum_{i=0}^{j-1} (1+pd)^i
  \end{pmatrix}.
  $$
  Thus, for $j=p^{n-1}$, we get $$\Id+\lambda_\xi+\lambda_\xi^2+\cdots +\lambda_\xi^{p^{n-1}-1}=0,$$
  because $\binom{p^{n-1}}{2}\equiv 0\pmod{p}$, and
  $$\sum_{i=0}^{p^{n-1}-1} (1+pd)^i=\frac{(1+pd)^{p^{n-1}}-1}{pd}\equiv 0\pmod{p^{n-1}}$$
  using \cite[Lemma 4]{rump4}.
  But this implies $\xi^{p^{n-1}}=(\Id+\lambda_\xi+\lambda_\xi^2+\cdots +\lambda_\xi^{p^{n-1}-1})(\xi)=0$, a
  contradiction with the fact that $\xi$ has multiplicative order equal to $p^n$.
\end{proof}

Fortunately, the complete classification of all the left brace structures with cyclic additive group of order $p^n$
is known \cite{rump4}. That allows us to determine all the basic solutions in the next example.

\begin{example}
  Now that we know that the additive group is always cyclic when the multiplicative group
  is $\Z/(p^n)$ with $n\geq 3$, we can apply Rump's classification \cite{rump4}.
  If the multiplicative group is also cyclic, there is a brace structure $G_i$ for every $i\in\{1,\dots,n\}$,
  given as a product over $\Z/(p^n)$ by $x\cdot y:=x+y+p^ixy$,
  except for $p=2$, where $i\in\{2,\dots,n\}$ (see \cite[Theorem 1]{rump4}). Note that, through all this example,
  we use a product with a dot to represent the multiplication in a brace, and a product without dot to represent
  the usual ring product over $\Z/(p^n)$.

  Then, it is easy to determine the socle of these braces $G_i$:
  $$
  \Soc(G_i)=\langle p^{n-i}\rangle\cong \Z/(p^i).
  $$
  We also need to determine the brace automorphism group:
  $$
  \Aut(G_i,+,\cdot)\cong\{1+p^{n-i}k:k\in \Z/(p^n)\}\leq (\Z/(p^n))^*,
  $$
  except when $i=n$, which is
  $$
   \Aut(G_n,+,\cdot)=\Aut(G_n,+)\cong (\Z/(p^n))^*
  $$

  To compute the orbits, observe that any element $a\in \Z/(p^n)$ can be written as
  $$
  a=a_0+a_1p+a_2p^2+\cdots +a_{n-1}p^{n-1},
  $$
  with $a_i\in\{0,1,\dots,p-1\}$, and the action of the lambda maps over one of these elements is
  \begin{eqnarray*}
  &&\lambda_x(a)=(1+p^ix)(a_0+a_1p+a_2p^2+\cdots +a_{n-1}p^{n-1})\\
  &=&(a_0+a_1p+\cdots+a_{i-1}p^{i-1})+(a_i+a_0x)p^i+\cdots+(a_{n-1}+a_{n-i-1}x)p^{n-1},
  \end{eqnarray*}
  for each $x\in \Z/(p^n)$. Note that if $a_0=\dots=a_{n-i-1}=0$, then $\lambda_x(a)=a$ for every $x$, so the orbit has only one
  element. Observe also that, if $k$ is the first integer between $0$ and $n-i-1$ such that $a_k=0$, then, using different $x$'s,
  any element of the form $a+mp^{i+k}$, $m\in\Z/(p^n)$, belongs to the orbit of $a$.
  Thus each orbit of $G_i$ with respect to $\lambda$ belongs to one of the following classes of orbits:
  \begin{enumerate}[(a)]

  \item $X^{(1)}_{\alpha}:=\{\alpha+p^ix: x\in \Z/(p^n)\}\text{, for each } \alpha\in\{1,\dots,p^i\},~ \alpha\not\equiv 0\pmod{p};$

  \item $X^{(2)}_{\beta p^k}:=\{\beta p^k+p^{i+k}x: x\in \Z/(p^n)\}$ for any $\beta\in\{1,\dots,p^i\},~ \beta\not\equiv 0\pmod{p},$
   and for any $k\in\{1,\dots,n-i-1\};$

  \item $X^{(3)}_{p^{n-i}\gamma}:=\{p^{n-i}\gamma\}$, for each $\gamma\in\{0,1,\dots,p^i-1\}.$

  \end{enumerate}

  For each class of orbits, the stabilizer of any element (which corresponds to the possible $H_i$'s) is equal to:
  \begin{enumerate}[(a)]
  \item Since the cardinality of the orbit is $\mid\{\alpha+p^ix:x\in \Z/(p^n)\} \mid=p^{n-i}$,
  the stabilizer is the unique subgroup of $\Z/(p^n)$ of order $p^i$;
  \item Since the cardinality of the orbit is $\mid \{\beta p^k+p^{i+k}x:x\in \Z/(p^n)\}\mid=p^{n-i-k}$
  the stabilizer is the unique subgroup of $\Z/(p^n)$ of order $p^{i+k}$;
  \item $G$.
  \end{enumerate}

  To have trivial intersection, we need at least one $K_{j,l}$ equal to $0$. Moreover, we need an additive
  generator of $\Z/(p^n)$ in the subset $X$, and all those elements belong to some $X^{(1)}_{\alpha}$.
  Thus the basic solutions are
  \begin{enumerate}
   \item $X_1=X^{(1)}_{\alpha}$, $K_{1,1}=0$.
   \item $X_1=X^{(3)}_{p^{n-i}\gamma}$, $X_2=X^{(1)}_{\alpha}$, $K_{1,1}=0$, $K_{2,1}=\langle p^{n-i}\rangle$.
   \item $X_1=X^{(2)}_{\beta p^k}$, $X_2=X^{(1)}_{\alpha}$, $K_{1,1}=0$, $K_{2,1}=\langle p^{n-i}\rangle$.
  \end{enumerate}

    Finally, to know if there is some isomorphism between these basic solutions, we need to apply brace
    automorphisms of $G$ over them. Note that, since
  the multiplicative group is abelian, the subgroups $K_{j,l}$ remain unchanged by the isomorphism of solutions,
  so we should only care about the effect over the subset $X$, which is equal to $X_1$ in case 1, and equal to $X_1\bigcup X_2$ in cases
  2 and 3. Thus two of these solutions are isomorphic if and only if
  they belong to the same class 1, 2 or 3, and the their two invariant subsets $X$ and $X'$ satisfies $X'=(1+p^{n-i}k)X$,
  for some $k\in\Z/(p^n)$, where here the product is the ring product of $\Z/(p^n)$.
\end{example}

On the other hand, the case of $Q_{2^m}$ is more difficult because we do not have a complete classification of
all the possible brace structures with multiplicative group isomorphic to $Q_{2^m}$.

\vspace{30pt}
 \noindent \begin{tabular}{llllllll}
 D. Bachiller && F. Ced\'{o}  \\
 Departament de Matem\`atiques &&  Departament de Matem\`atiques \\
 Universitat Aut\`onoma de Barcelona &&  Universitat Aut\`onoma de Barcelona  \\
08193 Bellaterra (Barcelona), Spain    && 08193 Bellaterra (Barcelona), Spain \\
 dbachiller@mat.uab.cat &&  cedo@mat.uab.cat\\
   &&   \\
E. Jespers &&  \\ Department of Mathematics &&
\\  Vrije Universiteit Brussel && \\
Pleinlaan 2, 1050 Brussel, Belgium &&\\
efjesper@vub.ac.be&&
\end{tabular}

\end{document}